\documentclass{article}
\usepackage[T1]{fontenc}
\usepackage[latin1]{inputenc}
\usepackage{amsfonts}
\usepackage{amssymb}
\usepackage{amsmath}
\usepackage{amsthm}
\usepackage{centernot}
\usepackage{enumerate}
\usepackage{indentfirst}
\usepackage{pgf,tikz}
\usepackage{lscape}
\usetikzlibrary{arrows}
\newcommand{\fracinline}[2]{\raisebox{0.4ex}{$#1$} / \raisebox{-0.7ex}{$#2$}}

\title{Friezes and continuant polynomials with parameters}
\author{V\'eronique Bazier-Matte\footnotemark[1] \and David Racicot-Desloges\footnotemark[1] \and Tanna Sanchez McMillan\footnotemark[1]}
\date{}
\begin{document}
\maketitle
\footnotetext[1]{D\'epartement de Math\'ematiques, Facult\'e des Sciences, Universit\'e de Sherbrooke,
Qu\'ebec, Canada, J1K 2R1}
\begin{abstract} Frieze patterns (in the sense of Conway and Coxeter [CC73]) are related to cluster algebras of type $\mathbb{A}$ and to signed continuant polynomials. In view of studying certain classes of cluster algebras with coefficients, we extend the concept of signed continuant polynomial to define a new family of friezes, called $c-$friezes, which generalises frieze patterns. Having in mind the cluster algebras of finite type, we identify a necessary and sufficient condition for obtaining periodic $c-$friezes. Taking into account the Laurent phenomenon and the positivity conjecture, we present ways of generating $c-$friezes of integers and of positive integers. We also show some specific properties of $c-$friezes. \end{abstract}
\section{Introduction}

Frieze patterns were introduced by Coxeter [Cox71] in 1971 and studied by Conway and Coxeter [CC73] in 1973. Here is an example of a frieze pattern over $\mathbb{Q}$: 

\[ \begin{array}{ccccccccccccccccccccccccc}
&& 0 && 0 && 0 && 0 && 0 && 0 && 0 && 0 && 0\\
& && 1 && 1 && 1 && 1 && 1 && 1 && 1 && 1\\
&& 2 && 1 && 3 && 1 && 2 && 2 && 1 && 3 && 1 \\
&\ldots && 1 && 2 && 2 && 1 && 3 && 1 && 2 && 2 && \ldots\\
&& 1 && 1 && 1 && 1 && 1 && 1 && 1 && 1 && 1 \\
& && 0 && 0 && 0 && 0 && 0 && 0 && 0 && 0
\end{array} \]

One can observe that it occupies a finite number (here $6$) of infinite rows in the discrete plane $\mathbb{Z}^2$. A frieze pattern over a field $K$ is a function $f$ that associates to specific points of the discrete plane $\mathbb{Z}^2$ elements of $K$, respecting the condition that every diamond shape

\[ \begin{array}{ccccccc}
&& f(a) && \\
& f(b) && f(c) \\
&& f(d) && 
\end{array} \]
\\
is such that $f(b)f(c)-f(a)f(d)=1$, together with the conditions that the first and last rows consist of zeros, while the second and penultimate rows consist of ones. A frieze pattern is said to be of order $n$ if it occupies $n+4$ rows in the discrete plane. Thus, in the above example, the frieze pattern is of order $n=2$.

It is shown in [CC73] that a frieze pattern is entirely determined by its elements in a diagonal or, alternatively, by its elements in the third row. Moreover, frieze patterns are always periodic. Let $x_i$, with $1 \leq i \leq n+3$, denote a family of $n+3$ consecutive elements in the third row of a frieze pattern. Conway and Coxeter showed in [CC73] that the value of every element in a frieze pattern is the evaluation of a signed continuant polynomial (see [AR12, Dup12]) in the variables $x_i$, signed continuant polynomials being a variant of the continuant polynomials of Euler (see [Wed32], p.133 and [Coh85], p.116). 

Our motivation for studying friezes comes from the theory of cluster algebras, introduced in 2002 by Fomin and Zelevinsky [FZ02] in order to understand the combinatorial aspects of total positivity [Lus98] and dual canonical bases [Lus90]. Cluster algebras were shown to have links with various mathematical areas, including -- but not limited to -- Lie theory [GLS13] and quiver representations [Kel10]. In [ARS10], Assem, Reutenauer and Smith use frieze patterns to calculate cluster variables. Indeed, let us consider a frieze pattern of order $n$ having a diagonal consisting of the initial cluster variables of a cluster algebra $\textit{A}$ of type $\mathbb{A}$. Then, the frieze pattern (without considering the rows of zeros and ones) consists of all the cluster variables of $\textit{A}$.

This article lays the groundwork for further study of cluster algebras of type $\mathbb{A}$. For instance, using the properties of the $c-$friezes here introduced, Fernando Borges proves in a forthcoming paper that the cluster variables of a $c-$cluster algebra are in bijection with the indecomposable objects of the cluster category associated to a quiver of type $\mathbb{A}_n$ [Bor13].

In this article, we first extend the concept of continuant polynomials of Euler to the $c-$continuant polynomials, which generalise the signed continuant polynomials (generalised Chebychev polynomials) of [CC73, AR12, Dup12]. We then proceed to the definition of a generalised version of frieze patterns (called \textit{$c-$friezes}), using the $c-$continuant polynomials. Afterwards, we study the periodicity of $c-$friezes, explaining why the frieze patterns of [CC73] are always periodic, while others never are. We also explain, in the case where the base field is $\mathbb{Q}$, when and how it is possible to obtain a $c-$frieze of integers and of positive integers, thus generalising results of Conway and Coxeter. We finally present transformations for obtaining $(-c)-$friezes from $c-$friezes and $c-$friezes of order $n+1$ from $c-$friezes of order $n$ and conversely.

\section{Continuant polynomials with parameters} 

Throughout the article, we let $K$ denote an arbitrary field and $\{x_i\}_{i \in \mathbb{Z}}$ a finite or countable set of indeterminates.

\subsection{}

\theoremstyle{definition}\newtheorem*{continuing}{Definition}
\begin{continuing}
Let $c \in K$ be a fixed nonzero scalar. The \textit{continuant polynomial with parameter $c$} (or, briefly, the \textit{$c-$continuant polynomial}) of order $k$ is defined by the recurrence relation 
$$P_k^c(x_1,x_2,\dots,x_k)=x_kP_{k-1}^c(x_1,x_2,\dots,x_{k-1})+cP_{k-2}^c(x_1,x_2,\dots,x_{k-2}) \text{ if } k\geq 1$$ together with the initial conditions $P_{-1}^c=0, P_0^c=1$.
\end{continuing}
For $k\geq 1$, the order $k$ of $P_k^c$ indicates the number of (always consecutive) indeterminates in the polynomial. Thus, for instance, we have

\begin{align*}
& P_1^c(x_1)=x_1 \\
& P_2^c(x_1,x_2)=x_1x_2+c \\
& P_3^c(x_1,x_2,x_3)=x_1x_2x_3+c(x_1+x_3).
\end{align*}

Whenever the parameter $c$ is fixed throughout a discussion, we simply write $P_k^c=P_k$.

\theoremstyle{remark}\newtheorem*{EulerCheb}{Remark}
\begin{EulerCheb}
The continuant polynomials of Euler are obtained by setting $c=1$, and the signed continuant polynomials (or generalised Chebyshev polynomials) of [CC73, AR12, Dup12] are obtained by setting $c=-1$.
\end{EulerCheb}

\theoremstyle{remark}\newtheorem*{remarkspolynomials}{Remark}
\begin{remarkspolynomials}
Let $b,c$ be two nonzero scalars and consider the polynomials $P_k^{\prime}$ given by the recurrence relation $P_k^{\prime}(x_1,x_2,\dots,x_k)=bx_kP_{k-1}^{\prime}(x_1,x_2,\dots,x_{k-1})+cP_{k-2}^{\prime}(x_1,x_2,\dots,x_{k-2})$ with the initial conditions $P_{-1}^{\prime}=0, P_0^{\prime}=1$. A straightforward induction on $k$ shows that $P_k^{\prime}(x_1,x_2,\dots,x_k)=P_k^c(bx_1,bx_2,\dots,bx_k)$. Thus, the study of $P_k^{\prime}$ reduces to the study of $P_k$.
\end{remarkspolynomials}
\subsection{}
We collect elementary properties of $c-$continuant polynomials in the following proposition, which generalises results in [CC73, AR12].

\theoremstyle{plain}\newtheorem*{Friezess}{Proposition}
\begin{Friezess}\label{Frisess}
Let $c,d$ be nonzero scalars and $k,l \geq 1$. Then,
\begin{enumerate}[a)]
\item $ P_k^c(x_1,x_2,\dots,x_k)=\begin{vmatrix}
 x_1    & -c      & 0       & \ldots & 0 & 0\\
 1      & x_{2} & -c      & \ldots & 0 & 0\\
 0      & 1       & x_{3} & \ldots & 0 & 0 \\
 \vdots & \vdots  & \vdots  & \ddots & \vdots & \vdots\\
 0      &0  & 0       & \ldots      & x_{k-1} & -c \\
  0      &0  & 0       & \ldots   & 1   & x_k 
 \end{vmatrix} $;
\item $P_k^c(x_1,x_2,\dots,x_k)=x_1P_{k-1}^c(x_2,x_3,\dots,x_k)+cP_{k-2}^c(x_3,x_4,\dots,x_k)$;
\item \begin{flushright} $P_{k+l}^c(x_1,x_2,\dots,x_{k+l}) = P_k^c(x_1,x_2,\dots,x_k)P_l^c(x_{k+1},x_{k+2},\dots,x_{k+l})  \hfill $ \\ $+cP_{k-1}^c(x_1,\dots,x_{k-1})P_{l-1}^c(x_{k+2},\dots,x_{k+l})$; \end{flushright} 
\item $\dfrac{P_k^c(x_1,...,x_{k})}{P_{k-1}^c(x_{2},...,x_{k})}=x_1+\dfrac{c}{x_{2}+\frac{c}{\dots + \frac{c}{x_{k}}}}$;
\item $P_k^c(x_1,x_2,...,x_{k})P_k^c(x_{2},x_3,...,x_{k+1})$\\
  $ \begin{array}{lll}
  && -P_{k-1}^c(x_2,x_3,...,x_{k})P_{k+1}^c(x_1,x_2,...,x_{k+1})=(-c)^k
  \end{array} $;
\item $P_k^{cd^2}(dx_1,\ldots dx_{k}) = d^kP_k^{c}(x_1,\ldots x_{k})$.
\end{enumerate}
\end{Friezess}

\begin{proof}
a) is done by induction on $k$ and implies b) upon developing along the first row. Then, one proves c) by induction on $l$, and d), e), f) by induction on $k$.
\end{proof}

\theoremstyle{remark}\newtheorem*{cpositive}{Remark}
\begin{cpositive}
When $K \subseteq \mathbb{R}$, if $x_1, x_2, \dots, x_{k}$ are strictly positive and $c>0$, then $$P_k^c(x_1,x_2,\dots,x_k)>0 ~.$$ It is easily seen from proposition~\ref{Frisess} a), by a straightforward induction on $k$.
\end{cpositive}

\subsection{} We need the following lemma in section 7. It says that, by switching the sign of $c$ and of one out of two indeterminates, the $c-$continuant polynomial is preserved, up to sign.

\theoremstyle{plain}\newtheorem*{UpToSigns}{Lemma}
\begin{UpToSigns}\label{Sign}
Let $\{x_j^{\prime}\}_{j \in \mathbb{Z}}$ be a finite or countable set of indeterminates, such that $x_{2l}^{\prime}=x_{2l}$ and $x_{2l+1}^{\prime}=-x_{2l+1}$ for all $l$. We then have that
\begin{enumerate}[a)]
\item if $k\equiv 0$ mod $4$, then, for all $j$, $P_k^{-c}(x_j^{\prime}, \dots, x_{j+k-1}^{\prime})=P_k^c(x_j,\dots,x_{j+k-1})$;
\item if $k\equiv 1$ mod $4$, then, \\
$
P_k^{-c}(x_j^{\prime}, \dots, x_{j+k-1}^{\prime}) = 
\begin{cases} 
P_k^c(x_j,\dots,x_{j+k-1}) & $for $j$ even$ \\ 
-P_k^c(x_j,\dots,x_{j+k-1}) & $for $j$ odd $
\end{cases}$;
\item if $k\equiv 2$ mod $4$, then, for all $j$, $P_k^{-c}(x_j^{\prime}, \dots, x_{j+k-1}^{\prime})=-P_k^c(x_j,\dots,x_{j+k-1})$;
\item if $k\equiv 3$ mod $4$, then, \\
$
P_k^{-c}(x_j^{\prime}, \dots, x_{j+k-1}^{\prime}) = 
\begin{cases} 
-P_k^c(x_j,\dots,x_{j+k-1}) & $for $j$ even$ \\ 
P_k^c(x_j,\dots,x_{j+k-1}) & $for $j$ odd $
\end{cases}$.
\end{enumerate}
\end{UpToSigns}

\begin{proof}
This is done by induction on $k$.
\end{proof}

\theoremstyle{remark}\newtheorem*{homogeneous}{Remark}
\begin{homogeneous}
This lemma would also follow from the observation that $P_k^c$ is homogeneous with respect to the grading which assigns degree $1$ to $c$ and the $x_{2l}$, and degree $0$ to the $x_{2l+1}$, and also homogeneous with respect to the the grading which assigns degree $1$ to $c$ and the $x_{2l+1}$, and degree $0$ to the $x_{2l}$.
\end{homogeneous}


\section{c-friezes}

The objective of this section is to introduce the notion of $c-$friezes and prove some of their properties that will be used in the sequel.

Throughout this section, the parameter $c$ is fixed and thus the notation $P_k$ is used, rather than $P^c_k$. Also, $\{x_i\}_{i \in \mathbb{Z}}$ is used to denote a set of elements in a field $K$. As we shall see in definition~\ref{frised}, this set is algebraically dependent.

\subsection{} 
For $n \geq 1$, we denote by $\mathbb{B}_n$ the set of all pairs $(i,j) \in \mathbb{Z}^2$ such that $-2 \leq j-i \leq n+1$. Such pairs are identified with the corresponding points in the discrete plane and one agrees to picture these pairs as follows:

\begin{center}
  $\cdots$
$
\begin{smallmatrix}
(i-1,i-3) 		& 		&		(i,i-2) 	& 		&		(i+1, i-1) 		& 		& 		(i+2,i)		&		&	\\
&		(i-1, i-2)		&		&		(i,i-1)		&		&		(i+1,i)		&		&		(i+2,i+1)	\\
(i-2,i-2)		&		&		(i-1,i-1)		&		&		(i,i)		&		&		(i+1,i+1)		&		&\\
&		(i-2,i-1)		&		&		(i-1,i)		&		&		(i,i+1)		&		&		(i+1,i+2) \\
\vdots		&		&		\vdots		&		&		\vdots		&		&		\vdots		&		&	\\
\\
\end{smallmatrix}
$
$\cdots$
\end{center}
We agree to call the set of all elements of the form $(i,i+k-1)$ the \textit{row $k$} of $\mathbb{B}_n$. Thus, $\mathbb{B}_n$ occupies $n+4$ rows in the discrete plane labeled from $-1$ to $n+2$. 
$ \\
\\ $
\indent We now proceed to define $c-$friezes.
\theoremstyle{definition}\newtheorem*{frised}{Definition}
\begin{frised}\label{frised}
\begin{enumerate}[a)]
$\qquad$\item Let $n \geq 1$. A set $\{x_i\}_{i \in \mathbb{Z}}$ is called \textit{$n-$admissible} if, for any $i \in \mathbb{Z}$, we have
$$P_{n+2}(x_i, \dots, x_{i+n+1})=0.$$
\item Let $n \geq 1$ and an $n-$admissible set $\{x_i\}_{i \in \mathbb{Z}}$. A \textit{$c-$frieze of order n} is a function $f:\mathbb{B}_n \rightarrow K$ such that, for all $(i,j) \in \mathbb{B}_n$, we have
$$f(i,j)=P_{j-i+1}(x_i, \dots, x_j).$$
\end{enumerate}
\end{frised}

From now on, we agree to designate the image of row $k$ of $\mathbb{B}_n$, that is elements of the form $f(i,i+k-1)$, by row $k$ of the $c-$frieze $f$.

\theoremstyle{remark}\newtheorem*{rown}{Remarks}
\begin{rown}
\begin{enumerate}[a)]
Let $f$  be a $c-$frieze of order $n$. Then,
\item For any $i$, we have $f(i,i-2)=P_{-1}=0$. Thus, in a $c-$frieze, both row $-1$ and row $n+2$ consist of zeros.
\item For any $i$, we have $f(i,i-1)=P_0=1$. Thus, in a $c-$frieze, row $0$ consists of ones.
\item For any $i$, we have $f(i,i)=P_1(x_i)=x_i$. Thus, in a $c-$frieze, row $1$ consists of the $x_i$.
\end{enumerate}
\end{rown}

\subsection{}
Since $c-$friezes are built using $c-$continuant polynomials, they share the same properties.
\theoremstyle{plain}\newtheorem*{frisep}{Lemma}
\begin{frisep}\label{frisep}
Let $i$, $j \in \mathbb{Z}$ be such that $0 \leq j-i \leq n+1$.
\begin{enumerate}[a)]
\item $f(i,j)=x_jf(i,j-1)+cf(i,j-2)$;
\item $f(i,j)=x_if(i+1,j)+cf(i+2,j)$;
\item $f(i,j)=
 \begin{vmatrix}
 x_i    & -c      & 0       & \ldots & 0 & 0\\
 1      & x_{i+1} & -c      & \ldots & 0 & 0\\
 0      & 1       & x_{i+2} & \ldots & 0 & 0 \\
 \vdots & \vdots  & \vdots  & \ddots & \vdots & \vdots \\
 0      & 0  & 0       & \ldots & x_{j-1} & -c  \\
 0      & 0  & 0       & \ldots & 1       & x_j 
 \end{vmatrix} $;
\item $ f(i,j-1)f(i+1,j) - f(i+1,j-1) f(i,j) = (-c)^{j-i}$.
\end{enumerate}
\end{frisep}

\begin{proof}
a) is the definition of $c-$continuant polynomial, while b), c) and d) follow from proposition~\ref{Frisess}, parts b), a) and e) respectively.
\end{proof}

Property d) is called the \textit{$(-c)-$modular rule}, or briefly the \textit{modular rule}. It is easily remembered by drawing a diamond consisting of 4 elements of the $c-$frieze, named a \textit{mesh}, arranged as follows:
\[
\begin{matrix}
&			&			f(i+1,j-1)	&			&\\
f(i,j-1)	&			&			&			f(i+1,j)\\
&			&			f(i,j)		&			\qquad \qquad \qquad .\\
\end{matrix}
\]

\theoremstyle{remark}\newtheorem*{j-i}{Remark}
\begin{j-i}
The exponent $j-i$ of $(-c)$ is also the index of the row containing $f(i,j-1)$ and $f(i+1,j)$.
\end{j-i}

\theoremstyle{remark}\newtheorem{exFrise}{Example}
\begin{exFrise}\label{exFrise}
The following example is a $4-$frieze of order $n=2$ over the field of rationals $\mathbb{Q}$.\\
$ \\ $
\resizebox{\linewidth}{!}{
$ \begin{array}{cccccccccccccccccc}
\ \ k=-1 	&& 0 && 0 && 0 && 0 && 0 && 0 && 0\\
	k=0 	&& &1 && 1 && 1 && 1 && 1 && 1\\
	k=1 	&& \fracinline{-14}{5} && 2 && -3 && -1 && \fracinline{4}{5} && \fracinline{35}{8} && \fracinline{-32}{25} \\
	k=2 	&\ldots && \fracinline{-8}{5} && -2 && 7 &&\fracinline{16}{5} &&\fracinline{15}{2} && \fracinline{-8}{5} && \ldots\\
	k=3 	&& 10 && \fracinline{-32}{5} && 10 && \fracinline{-32}{5} && 10 && \fracinline{-32}{5} && 10 \\
	k=4 	&& &0 && 0 && 0 && 0 && 0 && 0
\end{array} $}
\end{exFrise}

\subsection{}
In the previous example, row $3$ consists of the rational numbers $\fracinline{-32}{5}$ and $10$, appearing alternatively. One can observe that $(10) (\fracinline{-32}{5}) = (-4)^3$, the exponent $3$ being the number of the row and $4$ being the parameter of the frieze. This is a general fact.

\theoremstyle{plain}\newtheorem*{st}{Lemma}
\begin{st}\label{st}
Given a $c-$frieze of order $n$, there exist $s,t \in K$, such that the penultimate row, that is the row $n+1$ of the $c-$frieze, is of the form

$$
\begin{array}{cccccccccccc}
 \cdots & s & t & s & t & s & t & s & t &\cdots .
\end{array} 
$$

Furthermore, $st=(-c)^{n+1}$.
\end{st}

\begin{proof}
Applying the modular rule to the lowest mesh

$$
\begin{matrix}
&			&			f(i+1,i+n)	&			&\\
f(i,i+n)		&			&			&			f(i+1,i+n+1)\\
&			&			f(i,i+n+1)		&			&\\
\end{matrix}
$$
yields $(-c)^{n+1}=f(i,i+n)f(i+1,i+n)$, since $f(i,i+n+1)=0$. Set $s=f(i,i+n)$ and $t=f(i+1,i+n+1)$. Induction shows that, for any $k \in \mathbb{Z}$, $f(i+2k,i+n+2k)=s$ and $f(i+1+2k,i+n+1+2k)=t$.
\end{proof}

The notation $s,t$ for the two alternating elements of row $n+1$ will be adopted in the sequel. Without loss of generality, we shall assume $s=f(0,n)$ and $t=f(1,n+1)$. Thus, as shown before, we have $s=f(2i, n+2i)$ and $t=f(2i+1,n+2i+1)$ for all $i \in \mathbb{Z}$.

\subsection{}
Dually to lemma \ref{frisep}, it is also possible to express the elements of a $c-$frieze as a function of the elements appearing in the following rows, rather than the previous.

\theoremstyle{plain}\newtheorem*{lignen}{Lemma}
\begin{lignen} \label{rowk1} For $0 \leq k \leq n$,
\begin{enumerate}[a)]
\item $f(2i,2i+k-1) = \frac{f(2i-1,2i+k-1)f(2i,2i+n-1)}{t} + \frac{f(2i-2,2i+k-1)}{c}$;
\item $f(2i+1,2i+k) = \frac{f(2i,2i+k)f(2i+1,2i+n)}{s} + \frac{f(2i-1,2i+k)}{c}$;
\end{enumerate}
\end{lignen}

\begin{proof}
We prove a) by backward induction on $k$. For $k=n$, this follows from the facts that $f(2i-2,2i+n-1)=P_{n+2}(x_{2i-2}, x_{2i-1}, \dots, x_{2i+n-1})=0$ and $f(2i-1,2i+n-1)=t$. Because of the definition of the $c-$continuant polynomials, in order to apply induction, we need to consider the case $k=n-1$. By the modular rule, we have
$$f(2i-1,2i+n-2)f(2i,2i+n-1)-f(2i-1,2i+n-1)f(2i,2i+n-2)=(-c)^n.$$
Therefore, 
\[f(2i,2i+n-2)=\frac{f(2i-1,2i+n-2)f(2i,2i+n-1)-(-c)^n}{f(2i-1,2i+n-1)}.\]
Since $f(2i-1,2i+n-1)=t$ and $st=(-c)^{n+1}$, we find
\[f(2i,2i+n-2)=\frac{f(2i-1,2i+n-2)f(2i,2i+n-1)}{t}+\frac{s}{c}.\]
Since $s=f(2i-2,2i+n-2)$, we find
\[f(2i,2i+n-2)=\frac{f(2i-1,2i+n-2)f(2i,2i+n-1)}{t}+\frac{f(2i-2,2i+n-2)}{c}.\] \\
Assume it is true for $k$ and $k+1$. We shall prove it for $k-1$. Applying lemma \ref{frisep} a), we have
\[f(2i,2i+k)=x_{2i+k}f(2i,2i+k-1)+cf(2i,2i+k-2).\]
Therefore, 
\begin{align*}
& f(2i,2i+k-2) \\
&\qquad =\frac{f(2i,2i+k)}{c}-\frac{x_{2i+k}f(2i,2i+k-1)}{c} \\
&\qquad = \frac{f(2i-1,2i+k)f(2i,2i+n-1)}{tc}+\frac{f(2i-2,2i+k)}{c^2} \\
&\qquad \qquad -x_{2i+k}\frac{f(2i-1,2i+k-1)f(2i,2i+n-1)}{tc} \\
&\qquad \qquad -x_{2i+k}\frac{f(2i-2,2i+k-1)}{c^2} \\
&\qquad = \frac{f(2i,2i+n-1)}{t}\left[ \frac{f(2i-1,2i+k)-x_{2i+k}f(2i-1,2i+k-1)}{c}\right]  \\
&\qquad \qquad + \frac{1}{c}\left[ \frac{f(2i-2,2i+k)-x_{2i+k}f(2i-2,2i+k-1)}{c}\right]  \\
&\qquad = \frac{f(2i-1,2i+k-2)f(2i,2i+n-1)}{t} \\
&\qquad \qquad +\frac{f(2i-2,2i+k-2)}{c} \qquad.
\end{align*}
\item The proof of b) is similar and will be omitted.
\end{proof}

\section{Periodic $c-$friezes}

Cluster algebras of type $\mathbb{A}$ have a finite number of cluster variables (see [FZ03]). As a direct consequence, friezes of [CC73] are periodic. The objective of this section is to study the periodicity of $c-$friezes.

\subsection{}

\theoremstyle{definition}\newtheorem*{antiperiodic}{Definition}
\begin{antiperiodic} 
A $c-$frieze (or a row of a $c-$frieze) is called \textit{periodic} of \textit{period} $p$ if, for all elements $f(i,j)$ of this $c-$frieze (or this row, respectively), we have $f(i+p,j+p)=f(i,j)$ and $p$ is the least positive integer having this property. It is called \textit{antiperiodic} of \textit{period} $p$ if $f(i+p,j+p)=-f(i,j)$ and $p$ is the least positive integer having this property.
\end{antiperiodic}

Notice that if $f$ is antiperiodic of period $p$, then it is periodic of period $2p$.
\subsection{}
\theoremstyle{plain}\newtheorem*{transvection}{Transvection lemma}
\begin{transvection} \label{trans:box}
The elements of row $k$ are related to those of row $n-k+1$ by the following formulae:
\begin{enumerate}[a)]
\item $f(2i, 2i+k-1)=\frac{(-c)^{k}}{t}f(2i+k+1,2i+n+1)$ for all $i$;
\item $f(2i+1,2i+k)=\frac{(-c)^{k}}{s}f(2i+k+2,2i+n+2)$ for all $i$.
\end{enumerate}
\end{transvection}

\begin{proof} \
\begin{enumerate}[a)]
\item The statement is trivial for $k=-1$. For $k=0$, it follows from the facts that $f(2i,2i-1)=P_0=1$ and $f(2i+1,2i+n+1)=t$. Assume it is true for $k-1$ and $k$. We shall prove it for $k+1$. Applying lemma \ref{frisep} b), we have
$$f(2i+k,2i+n+1)=x_{2i+k}f(2i+k+1,2i+n+1)+cf(2i+k+2,2i+n+1).$$
Therefore,
\begin{align*}
&\frac{(-c)^{k+1}}{t}f(2i+k+2,2i+n+1) \\
& \qquad = \frac{(-c)^{k+1}}{ct}[f(2i+k,2i+n+1)-x_{2i+k}f(2i+k+1,2i+n+1)]\\
& \qquad = -(-cf(2i,2i+k-2)-x_{2i+k}f(2i,2i+k-1)\\
& \qquad = f(2i,2i+k)
\end{align*}
as required.
\item The proof is similar and will be omitted.
\end{enumerate}
\end{proof}

\subsection{} Unlike the friezes considered in [CC73], not all $c-$friezes are periodic. Our main result of this section says that rows of even order are always periodic and rows of odd order are periodic if $n$ is even.

\newtheorem*{periodicite}{Pseudo-periodicity theorem}
\begin{periodicite} \label{periodicity}
\ 
\begin{enumerate}[a)]
\item If row $k$ is even, then
$$f(2i, 2i+k-1)=f(2i+n+3,2i+k+n+2) \text{ for all } i \text{;} $$
$$f(2i+1,2i+k)=f(2i+n+4,2i+k+n+3) \text{ for all } i \text{.}$$

\item If row $k$ is odd, then
$$f(2i,2i+k-1)=\frac{(-c)^{n+1}}{t^2}f(2i+n+3,2i+k+n+2) \text{ for all } i \text{;} $$
$$f(2i+1,2i+k)=\frac{(-c)^{n+1}}{s^2}f(2i+n+4,2i+k+n+3) \text{ for all } i \text{.} $$

Furthermore, if $k$ is odd and $n$ is even, then
$$f(i,i+k-1)=f(i+2n+6,i+k+2n+5) \text{ for all } i \text{.} $$
\end{enumerate}
\end{periodicite}

\begin{proof}
\begin{enumerate}[a)] 
\item This is done by applying twice lemma \ref{trans:box}. At the second step, we use the fact that $k$ is even, writing $2j=2i+k$, so that, for all $i$:
\begin{align*}
f(2i,2i+k-1) &= \frac{(-c)^k}{t}f(2i+k+1,2i+n+1)\\
&= \frac{(-c)^k}{t}f(2j+1,2j+n-k+1)\\
&= \frac{(-c)^k}{t}\frac{(-c)^{n-k+1}}{s}f(2j+n-k+3,2j+n+2)\\
&= \frac{(-c)^{n+1}}{st}f(2i+n+3,2i+k+n+2)\\
&= f(2i+n+3,2i+k+n+2) \qquad.
\end{align*}
\item The proof of the first part of b) is similar and will be omitted. For the second part (when $n$ is even), we apply twice the first part of b).
\end{enumerate}
\end{proof}

\theoremstyle{remark}\newtheorem*{WhenPeriodic}{Remarks}
\begin{WhenPeriodic}
Note that all $c-$friezes with $n$ even are periodic. However, a $c-$frieze with $|s| \neq (-c)^{\frac{n+1}{2}} \neq |t|$ and with $n$ odd cannot be periodic. In fact, according to the first part of b) of the previous theorem, on rows of odd order, while half of the terms will always increase, the other half will always decrease.
\end{WhenPeriodic}

\subsection{}
\theoremstyle{plain}\newtheorem*{divisionn+3}{Corollary}
\begin{divisionn+3}
Rows of even order of a $c-$frieze have period a divisor of $n+3$, while rows of odd order have period a divisor of $2n+6$ if $n$ is even.
\end{divisionn+3}

\begin{proof}
For the first part, suppose that, contrary to our claim, the period $q$ of a row of even order is not a divisor of $n+3$. Then, by theorem~\ref{periodicity} a), we have $f(i,j)=f(i+q,j+q)=f(i+n+3,j+n+3)$. Let $p$ be the greatest common divisor of $q$ and $n+3$. Therefore, there exist integers $\alpha, \beta$, such that $p=\alpha q + \beta (n+3)$ and, hence, $f(i,j)=f((\alpha q + \beta (n+3))i, (\alpha q + \beta (n+3))j)$. Since $q$ is not a divisor of $n+3$, we have $p < q$, which is an absurdity.

We show the second part similarly, using $2n+6$ instead of $n+3$ and the second part of b) of theorem~\ref{periodicity}.
\end{proof}

\subsection{}
\theoremstyle{plain}\newtheorem*{consecelements}{Corollary}
\begin{consecelements}\label{consec}
A $c-$frieze is entirely determined by $n+3$ consecutive elements of the first row.
\end{consecelements}

\begin{proof}
Indeed, knowing $n+3$ consecutive elements of the first row yields all other elements of that row, following theorem \ref{periodicity}. The remaining elements follow from lemma \ref{frisep}.
\end{proof}

\subsection{}
\theoremstyle{plain}\newtheorem*{stper}{Corollary}
\begin{stper}\label{periodantiperiod}
\begin{enumerate}[a)] \qquad
\item A $c-$frieze with $s=t$ is periodic, with period a divisor of $n+3$.
\item A $c-$frieze with $s=-t$ is antiperiodic, with period a divisor of $n+3$.
\end{enumerate}
\end{stper}

\begin{proof}
a) and b) follow from the pseudo-periodicity theorem~\ref{periodicity}.
\end{proof}

\theoremstyle{remark}\newtheorem*{CCperiod}{Remark}
\begin{CCperiod}
Note that friezes in [CC73] are such that $s=t=1$ so they are periodic with a period which is a divisor of $n+3$.
\end{CCperiod}

\section{Sections}
Recall that in [CC73], it is shown that a frieze is entirely determined by its values in a "diagonal". Diagonals are a special case of sections (in the sense of [ASS06], p.302). We now show that a $c-$frieze is entirely determined by its values in any section, provided all are nonzero. This is a generalisation of [AD], p.2326, lemma 3.1).

\subsection{}
We start by recalling the definition of section.

\theoremstyle{definition}\newtheorem*{Section}{Definition}
\begin{Section} \label{Section}
A \textit{section} $\Sigma$ is a subset of $\mathbb{B}_n$ such that $|\Sigma|=n+4$ and such that if $(i_0,j_0) \in \Sigma$, then

\begin{enumerate}[a)]
\item either $(i_0, j_0-1)$ or $(i_0+1, j_0)$ is in $\Sigma$, and
\item either $(i_0-1, j_0)$ or $f(i_0, j_0+1)$ is in $\Sigma$
\end{enumerate}
(whenever these points are defined).

A section $\Sigma$ is \textit{oblique} (diagonal in the terminology of [CC73]) if all the elements of $\Sigma$ are of the form $(i_0,j)$ with $i_0$ fixed or $(i,j_0)$ with $j_0$ fixed.
\end{Section}

\theoremstyle{remark}\newtheorem{exSection}[exFrise]{Example}
\begin{exSection}

$$
\begin{matrix}
0\\
          & 1\\
&            & -3\\
          & -2\\
&         & 10\\
          & 0\\\\\\
\end{matrix}
$$
is the image of a non-oblique section by the $c-$frieze given in example~\ref{exFrise}. From now on, we refer to the image of a section by a $c-$frieze simply as a section of a $c-$frieze. Two oblique sections of the $c-$frieze of example~\ref{exFrise} are:
$$
\begin{matrix}
0\\
&1\\
&&2\\
&&&-2\\
&&&&10\\
&&&&&0
\end{matrix}
\qquad\text{ and  } \qquad
\begin{matrix}
&&&&&0\\
&&&&1\\
&&&\frac{4}{5}\\
&&\frac{16}{5}\\
&\frac{-32}{5}\\
0 & & & & & & & .\\
\end{matrix}
$$
\end{exSection}

\subsection{}
We now show how to find one of the original variables $x_i$ from three elements of an oblique section.

\theoremstyle{plain}\newtheorem*{CoefficientLinearisation}{Lemma}
\begin{CoefficientLinearisation}\label{coeff}
For all $i$, $j$, with $-1 \leq j-i \leq n$, if $f(i,j) \neq 0$, we have
\begin{enumerate}[a)]
\item $x_{j+1}=\dfrac{f(i,j+1)-cf(i,j-1)}{f(i,j)}$;
\item $x_{i-1}=\dfrac{f(i-1,j)-cf(i+1,j)}{f(i,j)}$.
\end{enumerate}
\end{CoefficientLinearisation}

\begin{proof}
We only prove a), because b) is similar. Because of lemma~\ref{frisep} a),
\begin{align*}
\frac{f(i, j+1) - cf(i, j-1)}{f(i, j)} &= \frac{x_{j+1}f(i, j) + cf(i,j-1) - cf(i, j-1)}{f(i, j)}\\
  &=x_{j+1}.\\
\end{align*}
\end{proof}

\subsection{}
\theoremstyle{plain}\newtheorem*{SectionFrise}{Proposition}
\begin{SectionFrise}
A $c$-frieze is entirely determined by a section of nonzero values.
\end{SectionFrise}

\begin{proof}
Let $\Sigma$ be a section of nonzero values. Fix $\Sigma=\{\sigma_{-1},\sigma_0,\dots,\sigma_{n+2}\}$ such that, for all $k \in \{-1,0,\dots,n+2\}$, $\sigma_k$ is in row $k$. Fix also $\Sigma^-=\{\sigma^-_{-1},\sigma^-_0,\dots,\sigma^-_{n+2}\}$ and $\Sigma^+=\{\sigma^+_{-1},\sigma^+_0,\dots,\sigma^+_{n+2}\}$ such that if $\Sigma_k=(i^*,i^*+k-1)$, then $\Sigma_k^-=(i^*-1,i^*+k-2)$ and $\Sigma_k^+=(i^*+1,i^*+k)$. Since all elements in row $0$ and in row $n+1$ (because of lemma~\ref{st}) are known, the image of $\Sigma^-$ and $\Sigma^+$ are recursively deduced from the modular rule. Finally, let $x_i=f(\sigma_1)$.

We want to show by induction that for all $m$ such that $0 \leq m \leq n+1$, there exist positive integers $k_m$ and $l_m$ with:
\begin{itemize}
\item $k_m+l_m=m$,
\item $k_{m+1}=k_m$ or $l_{m+1}=l_m$,
\item $(i-k_m,i+l_m) \in \Sigma$ and
\item both $x_{i-k_m-1}$ and $x_{i+l_m+1}$ are known.
\end{itemize}

For $k+l=0$, there exist $k=l=0$ such that $f(i,i)=x_i \in f(\Sigma)$. Furthermore, since $x_{i-1}=f(i-1,i-1)$ and $x_{i+1}=f(i+1,i+1)$ are respectively in the image of $\Sigma^-$ and $\Sigma^+$, they are both known.

Suppose that there exists $M \in \{0,1,\dots,n\}$ such that for all $m \in \{0,1,\dots,M\}$, there exist $k_m,l_m \in \mathbb{N}$ such that:
\begin{itemize}
\item $k_m+l_m=m$,
\item $k_{m+1}=k_m$ or $l_{m+1}=l_m$ (except for $m=M$),
\item $(i-k_m,i+l_m) \in \Sigma$ and 
\item both $x_{i-k_m-1}$ and $x_{i+l_m+1}$ are known.
\end{itemize} Then, following definition \ref{Section}, either $(i-k_M-1, i+l_M) \in \Sigma$ or $f(i-k_M, i+l_M+1) \in \Sigma$. If $(i-k_M-1, i+l_M) \in \Sigma$, then $f(i-k_M-2, i+l_M)$ is known, since it is in $\Sigma^-$ or $\Sigma$. Because of proposition \ref{coeff} b),

\[\frac{f(i-k_M-2, i+l_M)-c f(i-k_M, i+l_M)}{f(i-k_M-1, i+l_M)} = x_{i-k_M-2}~.\]

Therefore, $k_{M+1}=k_M+1$ and $l_{M+1}=l_M$ are such that $(i-k_M-1,i+l_M) \in \Sigma$ and both $x_{i-k_M-2}$ and $x_{i+l_M+1}$ are known. Similarly, if $(i-k_M, i+l_M+1) \in \Sigma$, then $x_{i+l_M+2}$ is known, using proposition~\ref{coeff} a). Then, $k_{M+1}=k_M$ and $l_{M+1}=l_M+1$ are such that $(i-k_M,i+l_M+1) \in \Sigma$ and both $x_{i-k_M-1}$ and $x_{i+l_M+2}$ are known.

Thus, using this reasoning, there exists $n+4$ consecutive elements in the first row that are known. Hence, because of theorem~\ref{periodicity} and lemma~\ref{frisep} c), all the $c-$frieze is known.
\end{proof}

\section{Friezes of positive integers}
Throughout this section, we fix $K=\mathbb{Q}$. Having in mind the Laurent phenomenon (see [FZ02]) and the positivity conjecture (see [FZ03]) in cluster algebras, we wish to find a sufficient condition for having a $c-$frieze of integers and even a $c-$frieze of positive integers. Recall that in [CC73], for having a $c-$frieze of positive integers, it is sufficient to have an oblique section $f(i_0,j)$ consisting of strictly positive elements, with $$f(i_0, j)|f(i_0, j+1)+ f(i_0, j-1) \text{ for all } i_0,j, \text{ such that } 0 \leq j- i_0 \leq n+1.$$ This condition is not sufficient for $c-$friezes.


\subsection{}
Following the remark of lemma~\ref{Frisess}, if $c>0$, it is not possible to have a $c-$frieze that consists only of positive integers. However, if $c<0$, it is possible.

\theoremstyle{plain}\newtheorem*{SectionFrisePositive}{Proposition}
\begin{SectionFrisePositive} \label{positif}
Let $f$ be a $c-$frieze with $c < 0$. If a section of $f$ consists only of positive elements, then the entire $c-$frieze consists only of positive elements.
\end{SectionFrisePositive}

\begin{proof}
Let $\Sigma$ be that section. Since rows $-1$ and $0$ are known, consider elements in row 1. Draw between any 2 neighbouring points of $\Sigma$ an arrow, always oriented from left to right. Then, we obtain a quiver with underlying graph $\mathbb{A}$. In particular, it has a source and a sink. If a source $f(i,j)$ is given, then $f(i+1,j+1)$ is immediately deduced from the values in $\Sigma$ and the modular rule. Furthermore, since $c$ is negative and elements in $\Sigma$ are positive, using the modular rule, we have $f(i+1,j+1)=\frac{(-c)^{j-i+1}+f(i+1,j)f(i,j+1)}{f(i,j)} > 0$. Repeating this reasoning with the new section $\Sigma ^\prime = (\Sigma \diagdown \{f(i,j)\}) \cup \{f(i+1,j+1)\}$, we get another source and, inductively, we obtain all elements on the right of $\Sigma$. They are all positive. Applying the same reasoning to sinks, we obtain all elements on the left of $\Sigma$ and they are also positive.
\end{proof}

\subsection{}
Since $c-$friezes with $|s|=|t|$ are periodic, regardless of their order, we will only be considering these $c-$friezes for the remaining of this section. We call such $c-$friezes \textit{monotonic $c-$friezes}, or briefly \textit{monotonic friezes}.

To ensure that $|s|=|t|$, it is necessary that $(-c)^{\frac{n+1}{2}} \in \mathbb{Q}$ (which implies that $|s|=|(-c)^{\frac{n+1}{2}}|$). This is always the case if $n$ is odd. It is easily seen that it is necessary that $(-c)^{\frac{n+1}{2}} \in \mathbb{Z}$ to have a $c-$frieze of integers. This is always the case if $n$ is odd and $c \in \mathbb{Z}$.

\theoremstyle{plain}\newtheorem*{integers}{Frieze of integers theorem}
\begin{integers} \label{entier}
Let $f$ be a monotonic frieze with $s \in \mathbb{Z}$ (that is if and only if $t \in \mathbb{Z}$). The three following conditions are equivalent:
\begin{enumerate}[a)]
\item The frieze consists only of integers.
\item There exists an oblique section such that
  \begin{enumerate}[i)]
          \item for all $-1 \leq j- i_0 \leq n$, we have $f(i_0, j)|f(i_0, j+1)-c f(i_0, j-1)$ and
    \item
      $ \dfrac{c}{|s|} \begin{vmatrix}
      \mu_0 & -c & 0 & \ldots & 0 & 0\\
      1 & \mu_1 & -c & \ldots & 0 & 0\\
      0 & 1 & \mu_2 & \ldots & 0 & 0\\
      \vdots & \vdots & \vdots & \ddots & \vdots & \vdots\\
      0 & 0 & 0 & \ldots & \mu_{n-2} & -c \\
      0 & 0 & 0 & \ldots & 1 & \mu_{n-1} 
      \end{vmatrix}~\in~\mathbb{Z},$ $ \\ \\ $
      where $\mu_l = \dfrac{f(i_0,i_0+l)-c f(i_0,i_0+l-2)}{f(i_0,i_0+l-1)}$.
  \end{enumerate}
\item There exists an oblique section such that
  \begin{enumerate}[i)]
          \item for all $-1 \leq i- j_0 \leq n$, we have $f(i, j_0)|f(i-1, j_0)-c f(i+1,j_0)$ and
    \item
      $ \dfrac{c}{|s|} \begin{vmatrix}
      \nu_0 & -c & 0 & \ldots & 0 & 0\\
      1 & \nu_1 & -c & \ldots & 0 & 0\\
      0 & 1 & \nu_2 & \ldots & 0 & 0\\
       \vdots &  \vdots &  \vdots &  \ddots &  \vdots & \vdots\\
      0 & 0 & 0 & \ldots & \nu_{n-2} & -c \\
      0 & 0 & 0 & \ldots & 1 & \nu_{n-1} 
      \end{vmatrix}~\in~\mathbb{Z},$ $ \\ \\ $
      where $\nu_l = \dfrac{f(j_0-l,j_0)-c f(j_0-l+2,j_0)}{f(j_0-l+1,j_0)}$.
  \end{enumerate}
\end{enumerate}

\end{integers}
\begin{proof}
Following lemma~\ref{coeff}, a) trivially implies b) and c). Assume that the conditions of b) are satisfied. Then, $f(i_0,i_0+k-1) \neq 0$ for all $k$ such that $0 \leq k \leq n+1$. \\
By lemma~\ref{coeff} a), $$x_{i_0+k}=\frac{f(i_0,i_0+k)- c f(i_0,i_0+k-2)}{f(i_0,i_0+k-1)}.$$
Therefore, by hypothesis, $$x_{i_0+k} \in \mathbb{Z} \text{ for all } k \text{ such that }0\leq k \leq n+1.$$
Again by lemma~\ref{coeff} a), 
\[x_{i_0+n+2}=\frac{f(i_0,i_0+n+2)-c f(i_0,i_0+n)}{f(i_0,i_0+n+1)}.\]
Since $f(i_0,i_0+n+1)=\pm |s|$ and $f(i_0,i_0+n+2)=0$, we have
\[x_{i_0+n+2}=\pm\frac{c f(i_0,i_0+n)}{|s|}.\]
By lemma \ref{frisep} c) and lemma \ref{coeff} a), we have
\[x_{i_0+n+2}=\pm
\frac{c}{|s|} \begin{vmatrix}
      \mu_0 & -c & 0 & \ldots & 0 & 0\\
      1 & \mu_1 & -c & \ldots & 0 & 0\\
      0 & 1 & \mu_2 & \ldots & 0 & 0\\
      \vdots & \vdots & \vdots & \ddots & \vdots & \vdots\\
      0 & 0 & 0 & \ldots & \mu_{n-2} & -c \\
      0 & 0 & 0 & \ldots & 1 & \mu_{n-1} 
      \end{vmatrix}\]

which is an integer, by hypothesis.\\
Since $|s|=|t|$, by corollary~\ref{periodantiperiod}, $x_{i_0+k}\in \mathbb{Z} \text{ for all } k.$

Finally, following lemma \ref{frisep}, all the frieze consists of integers.

One can show similarly that c) implies a).
\end{proof}

\theoremstyle{remark} \newtheorem*{equivalent}{Remarks}
\begin{equivalent} \ 
\begin{enumerate}[a)]
\item Note that the second condition of b) is equivalent to having $\frac{c f(i_0+1,i_0+n)}{|s|}=\pm f(i_0+n+2,i_0+n+2) \in \mathbb{Z}$ and the second condition of c) is equivalent to having $\frac{c f(j_0-n,j_0-1)}{|s|}=\pm f(j_0+1,j_0+1) \in \mathbb{Z}$, using lemma~\ref{trans:box} and theorem~\ref{periodicity}. Therefore, a monotonic frieze consists only of integers if and only if its first row only consists of integers.
\item Note that friezes in [CC73] are such that the additional condition of b) (or of c), respectively) is always satisfied, since $s=1$.
\end{enumerate}
\end{equivalent}

\theoremstyle{remark}\newtheorem{Exampleintegers}[exFrise]{Example}
\begin{Exampleintegers}
As an example, consider the following $(-4)-$frieze of order $n=2$ with $|s|=|t|=8$:
$$ $$
\[\ldots
\begin{array}{cccccccccccccccccccc}
& 0 && 0 && 0 && 0 && 0 && 0 && 0 && 0 && 0 &\\
 && 1 && 1 && 1 && 1 && 1 && 1 && 1 && 1 &\\
& 4 && 3 && 3 && 4 && \frac{5}{2} && 4 && 3 && 3 && 4 &\\
&& 8 && 5 && 8 && 6 && 6 && 8 && 5 && 8 &\\
& 8 && 8 && 8 && 8 && 8 && 8 && 8 && 8 && 8 &\\
&& 0 && 0 && 0 && 0 && 0 && 0 && 0 && 0 &
\end{array} \ldots\]
$ \\
\\ $
While the oblique section
\[
\begin{matrix}
0\\
&1\\
&&4\\
&&&8\\
&&&&8\\
&&&&&0
\end{matrix}
\]
satisfies the condition $$f(i_0, j)|f(i_0, j+1)-c f(i_0, j-1) \text{ for all } i_0,j \text{, such that }0 \leq j- i_0 \leq n+1,$$ the corresponding monotonic frieze does not consist only of integers.
\end{Exampleintegers}

\subsection{}

\theoremstyle{plain}\newtheorem*{naturel}{Corollary}
\begin{naturel}
Let $c<0$. If a monotonic $c-$frieze has an oblique section of positive integers satisfying the conditions of the friezes of integers theorem, then the entire frieze consists of positive integers.
\end{naturel}

\begin{proof}
Directly follows from proposition~\ref{positif} and friezes of integers theorem~\ref{entier}.
\end{proof}

\subsection{}
Another type of monotonic friezes is such that it is sufficient (and necessary) to have only the first condition of the friezes of integers theorem~\ref{entier} of b) (or of c), respectively).

\theoremstyle{definition}\newtheorem*{AltFr}{Definition}
\begin{AltFr}
A monotonic frieze over $\mathbb{Q}$ is called \textit{alternating} if $c=1$.
\end{AltFr}

Note that since $c=1$, we have $|s|=|t|=1$. 

The $(-c)-$modular rule, applied to these friezes, can be written as the determinant:

$$\left|
\begin{array}{cc}
f(i,j-1) & f(i,j)\\
f(i+1,j-1) & f(i+1,j)\\
\end{array}
\right| = (-1)^{j-i}.$$

Thus, since the determinant is $\pm 1$, the corresponding matrices are invertible and belong to the general linear group $GL_2(\mathbb{Q})$.

One can show that the second condition of the friezes of integers theorem~\ref{entier} of b) (or of c), respectively) is always satisfied, since $|s|=1$.

\theoremstyle{remark}\newtheorem{ExampleAlt}[exFrise]{Example}
\begin{ExampleAlt}
As an example, consider the oblique section of an alternating frieze of order $n=4$\[
\begin{matrix}
0\\
&1\\
&&-4\\
&&&-3\\
&&&&5\\
&&&&&2\\
&&&&&&-1\\
&&&&&&&0 \qquad
\end{matrix}
\]
that satisfies the condition $$f(i_0, j)|f(i_0, j+1)- f(i_0, j-1) \text{ for all }j \text{, such that }0 \leq j- i_0 \leq n+1.$$
Its corresponding frieze consists only of integers:
$$ $$
\resizebox{\linewidth}{!}{
$ \begin{array}{ccccccccccccccccccc}
&& 0 && 0 && 0 && 0 && 0 && 0 && 0 && 0\\
&& & 1 && 1 && 1 && 1 && 1 && 1 && 1\\
&& -4 && 1 && -3 && 1 && -3 && 1 && -4 && 1 \\
& && -3 && -2 && -2 && -2 && -2 && -3 && -3\\
& \ldots & -2 && 5 && -1 && 3 && -1 && 5 && -2 && 5 & \ldots\\
&& & 3 && 2 && 1 && 1 && 2 && 3 && 3 \\
&& -1 && 1 && -1 && 1 && -1 && 1 && -1 && 1 \\
&& & 0 && 0 && 0 && 0 && 0 && 0 && 0
\end{array} $}
$$ $$
\end{ExampleAlt}

We will see in the next section that alternating friezes are related to the friezes of [CC73].

\section{Transformation of $c-$friezes}

Until now, we studied friezes having both $c$ and $n$ fixed. The objective of this section is to discuss ways of generating friezes from others, using transformations. To do so, either the parameter $c$ or the order $n$ will change.

\subsection{}

\theoremstyle{definition}\newtheorem*{o-mega}{Definition}
\begin{o-mega}
We call \textit{$c-$polygonal sequence of order $n$} a family of $n+3$ consecutive elements of the first row of a $c-$frieze of order $n$.
We denote it by $$\psi_c^n=(x_1, x_2, x_3, ..., x_{n+3}).$$
We denote the set of all $c-$polygonal sequences of order $n$ by $\Psi_c^n$.
\end{o-mega} 

\theoremstyle{remark}\newtheorem*{Consecutive}{Remark}
\begin{Consecutive}
Since a $c-$frieze is entirely determined by $n+3$ consecutive elements in its first row, following theorem~\ref{periodicity}, there exists one and only one $c-$frieze associated with a $c-$polygonal sequence of order $n$. However, there are generally more than one (possibly a countable infinity) $c-$polygonal sequences of order $n$ corresponding to a $c-$frieze of order $n$.
\end{Consecutive}

\subsection{}
\theoremstyle{plain}\newtheorem*{c-c}{Lemma}
\begin{c-c}\label{c-c}
There is a bijection between the set of $c-$polygonal sequences of order $n$ and the set of $(-c)-$polygonal sequences of order $n$.
\end{c-c}

\begin{proof}
Consider $$
\begin{array}{rccc}
g : & \Psi_c^n & \rightarrow & \Psi_{-c}^{n}\\
& (x_1, x_2, x_3, \dots, x_{n+3}) & \mapsto & (-x_1, x_2, -x_3, \dots, (-1)^{n+3}x_{n+3})
\end{array}
$$
that is switching the sign of one out of two indeterminates, starting with the first one.
By lemma~\ref{Sign}, the image of this transformation is a $(-c)-$polygonal sequence of order $n$. The function is involutive, so it is bijective.
\end{proof}

\subsection{}
\theoremstyle{plain}\newtheorem*{c-fr-c-fr}{Theorem}
\begin{c-fr-c-fr}\label{c-fr-c-fr}
For any $c-$frieze $f$ of order $n$, there exists a $(-c)-$frieze $f^{\prime}$ of order $n$ such that

$
f'(i,i+k-1) = 
\begin{cases} 
  &$ if $k\equiv 0$ mod $4$; $\\
f(i, i+k-1)& \ $if $ k\equiv 1$ mod $4$ and $i$ is even; $\\
& \ $if $ k\equiv 3$ mod $4$ and $i$ is odd;$ \\
\\
&$ if $k\equiv 1$ mod $4$ and $i$ is odd; $\\
-f(i,i+k-1) & \ $if $ k\equiv 2$ mod $4$; $\\
& \ $if $ k\equiv 3$ mod $4$ and $i$ is even.$
\end{cases}$ 
\end{c-fr-c-fr}

\begin{proof}
This directly follows from the proof of lemma~\ref{c-c} and from lemma~\ref{Sign}.
\end{proof}

\theoremstyle{remark}\newtheorem*{UpToSign}{Remarks}
\begin{UpToSign}
\begin{enumerate}[a)] \ 
\item Note that there are generally two $(-c)-$friezes of order $n$ generated from a $c-$frieze of order $n$, by using the method of theorem~\ref{c-fr-c-fr}.
\item Note that friezes $f$ and $f^{\prime}$ are the same, up to sign.
\item As a direct consequence of theorem~\ref{c-fr-c-fr}, it is possible to construct all alternating friezes of nonzero integers from friezes of [CC73] and vice-versa.
\item As a direct consequence of theorem~\ref{c-fr-c-fr}, for any triangulation of a convex polygon (in the sense of [CC73]), there exists an alternating frieze of integers.
\end{enumerate}
\end{UpToSign}

\subsection{}
Since there are generally more than one $c-$polygonal sequence of order $n$ associated with a given $c-$frieze, it is natural to consider two $c-$polygonal sequences of the same frieze as equivalent. We denote by $\sim$ the equivalence relation in $\Psi_c^n$ defined by $\psi_{c,1}^n \sim \psi_{c,2}^n$ if and only if $\psi_{c,1}^n$ and $\psi_{c,2}^n$ are associated to the same $c-$frieze. Then, we let $\widetilde{\Psi}_c^n=\Psi_c^n/\!\raisebox{-.65ex}{\ensuremath{\mathcal{\sim}}} \ \ $. It is immediate that there is a bijection between the set of $c-$friezes of order $n$ and $\widetilde{\Psi}_c^n$.

\theoremstyle{plain}\newtheorem*{bbc2th}{Proposition}
\begin{bbc2th}\label{cd2}
Let $d \in K^*$. There is a bijection between the set of $c-$friezes of order $n$ and the set of $cd^2-$friezes of order $n$. Furthermore, this bijection maps $c-$friezes of integers injectively to $cd^2$-friezes of integers.
\end{bbc2th}

\begin{proof}
Consider $$
\begin{array}{rccc}
h : & \widetilde{\Psi}_c^n & \rightarrow & \widetilde{\Psi}_{cd}^n\\
& (x_1, x_2, x_3, \dots, x_{n+2}, x_{n+3}) \qquad & \mapsto & (dx_1, dx_2, dx_3, \dots, dx_{n-2}, dx_{n+3}).
\end{array}
$$

By lemma~\ref{Frisess} f), the image of this transformation is a $cd^2-$polygonal sequence of order $n$. Its inverse function is:
$$
\begin{array}{rccc}
h^{-1} : & \widetilde{\Psi}_{cd}^n & \rightarrow & \widetilde{\Psi}_c^n\\
& (x_1, x_2, x_3, \dots, x_{n+2}, x_{n+3}) \qquad & \mapsto & (\frac{x_1}{d}, \frac{x_2}{d}, \frac{x_3}{d}, \dots, \frac{x_{n-2}}{d}, \frac{x_{n+3}}{d})
\end{array}
$$
so it is bijective.

The second part follows from proposition~\ref{Frisess} f).
\end{proof}

\subsection{}
Recall that in [CC73], it is shown how to generate all friezes of order $n+1$ from friezes of order $n$ and vice-versa. We extend this technique to monotonic friezes with $s=t$ and $c<0$. Such monotonic friezes are called \textit{repetitive}. We denote by $\Upsilon_c^n$ the set of all $c-$polygonal sequences of repetitive $c-$friezes of order $n$. Note that $\Upsilon_c^n$ is a proper subset of $\Psi_c^n$.

\theoremstyle{definition}\newtheorem*{c-induced}{Definition}
\begin{c-induced}
A repetitive frieze is called \textit{$c-$induced} if it has an element of value $(-c)^{1/2}$ in its first row. We denote by $\Phi_c^n$ the set of all $c-$polygonal sequences of $c-$induced friezes of order $n$. Since $\Phi_c^n \subset \Psi_c^n$, we let $\widetilde{\Phi}_c^n=\Phi_c^n/\!\raisebox{-.65ex}{\ensuremath{\mathcal{\ \sim}}} \ \ $.
\end{c-induced}

\theoremstyle{remark}\newtheorem*{fieldQ}{Remark}
\begin{fieldQ}
All friezes of positive integers in [CC73] are $(-1)-$induced. A proof is given in [CC73], p.180.
\end{fieldQ}

\subsection{}
\theoremstyle{plain}\newtheorem*{nn+1}{Lemma}
\begin{nn+1}\label{nn+1}
There exists a surjective transformation from the set of all $c-$polygonal sequences of repetitive friezes of order $n$ to the set of all equivalent $c-$polygonal sequences of $c-$induced friezes of order $n+1$.
\end{nn+1}

\begin{proof}
Consider
\[
\begin{array}{r r c l}
\Gamma : & \Upsilon_c^n & \rightarrow & \widetilde{\Phi}_c^{n+1}\\
& (x_1, x_2, x_3, \dots, x_{n+2}, x_{n+3}) & \mapsto & (x_1+(-c)^{1/2}, x_2, x_3, \dots,  \\
& & & \qquad \qquad x_{n+2}, x_{n+3}+(-c)^{1/2},(-c)^{1/2})
\end{array}
\]

We first show that the image of a $c-$polygonal sequence of a repetitive frieze of order $n$ by $\Gamma$ is a sequence of elements of the first row of a $c-$induced frieze of order $n+1$. It is trivial that there is an element of value $(-c)^{\frac{1}{2}}$ in the first row. Then, we must show that for all $i$,
\[P_{n+2}(x_i, x_{i+1}, \dots, x_{n+3}+(-c)^{1/2},(-c)^{1/2},x_1+(-c)^{1/2}, x_2, \dots, x_{i-3})=(-c)^{1/2}s\]
and 
\[P_{n+3}(x_i,x_{i+1}, \dots, x_{n+3}+(-c)^{1/2},(-c)^{1/2},x_1+(-c)^{1/2}, x_2, \dots, x_{i-2})=0~.\]
Since $(x_1, x_2, x_3, \dots, x_{n+3})$ is a sequence of elements of the first row of a repetitive $c$-frieze of order $n$, for all $i$ we have 
\[P_{n+1}(x_i, x_{i+1}, \dots, x_{i+n})=P_{n+1}(x_i, x_{i+1}, \dots, x_{i-3})=s \text{ and}\]
\[P_{n+2}(x_i, x_{i+1}, \dots, x_{i+n+1})=P_{n+2}(x_i, x_{i+1}, \dots, x_{i-2})=0~.\]

According to proposition~\ref{Frisess} c), we have
\begin{align*}
&P_{n+2}(x_{i},...,x_{n+3}+(-c)^{\frac{1}{2}},(-c)^{\frac{1}{2}},x_{1}+(-c)^{\frac{1}{2}},...,x_{i-3}) \\
 & \qquad =P_{n-i+3}(x_i,...,x_{n+2})P_{i-1}(x_{n+3}+(-c)^{\frac{1}{2}},(-c)^{\frac{1}{2}},x_{1}+(-c)^{\frac{1}{2}},...,x_{i-3})\\
 & \qquad \qquad +cP_{n-i+2}(x_i,...,x_{n+1})P_{i-2}((-c)^{\frac{1}{2}},x_{1}+(-c)^{\frac{1}{2}},...,x_{i-3})~.
\end{align*}
Then, using proposition~\ref{Frisess} b) twice, we have
\begin{align*}
&P_{n+2}(x_{i},...,x_{n+3}+(-c)^{\frac{1}{2}},(-c)^{\frac{1}{2}},x_{1}+(-c)^{\frac{1}{2}},...,x_{i-3})\\
& \qquad =P_{n-i+3}(x_i,...,x_{n+2}) [(x_{n+3}+(-c)^{\frac{1}{2}})P_{i-2}((-c)^{\frac{1}{2}},x_{1}+(-c)^{\frac{1}{2}},...,x_{i-3})\\
& \qquad \qquad +cP_{i-3} (x_{1}+(-c)^{\frac{1}{2}},...,x_{i-3})]\\
& \qquad \qquad +cP_{n-i+2}(x_i,...,x_{n+1})[(-c)^{\frac{1}{2}}P_{i-3}(x_{1}+ (-c)^{\frac{1}{2}},...,x_{i-3})\\
& \qquad \qquad +cP_{i-4}(x_{2},...,x_{i-3})] \\
& \qquad =P_{n-i+3}(x_i,...,x_{n+2})[(x_{n+3}+(-c)^{\frac{1}{2}})[(-c)^{\frac{1}{2}}P_{i-3}(x_{1}+(-c)^{\frac{1}{2}},...,x_{i-3}) \\
& \qquad \qquad +cP_{i-4}(x_{2},...,x_{i-3})]+cP_{i-3}(x_{1}+(-c)^{\frac{1}{2}},...,x_{i-3})]\\
& \qquad \qquad +cP_{n-i+2}(x_i,...,x_{n+1})[(-c)^{\frac{1}{2}}[(x_{1}+(-c)^{\frac{1}{2}})P_{i-4}(x_{2},...,x_{i-3})\\
& \qquad \qquad +cP_{i-5}(x_{3},...,x_{i-3})] +cP_{i-4} (x_{2},...,x_{i-3})]~, \\
\end{align*}
which, by developing, becomes
\begin{align*}
&P_{n+2}(x_{i},...,x_{n+3}+(-c)^{\frac{1}{2}},(-c)^{\frac{1}{2}},x_{1}+(-c)^{\frac{1}{2}},...,x_{i-3})\\
& \qquad=P_{n-i+3}(x_i,...,x_{n+2})[(-c)^{\frac{1}{2}}x_{n+3}P_{i-3}(x_{1}+(-c)^{\frac{1}{2}},...,x_{i-3})\\
& \qquad \qquad -(-c)^{\frac{3}{2}}P_{i-4}(x_{2},...,x_{i-3})+cx_{n+3}P_{i-4}(x_{2},...,x_{i-3})]\\
& \qquad \qquad + cP_{n-i+2}(x_i,...,x_{n+1})[(-c)^{\frac{1}{2}}x_{1}P_{i-4}(x_{2},...,x_{i-3})\\
& \qquad \qquad -(-c)^{\frac{3}{2}}P_{i-5} (x_{3},...,x_{i-3})]~.
\end{align*}
Using once again proposition~\ref{Frisess}b), we find
\begin{align*}
&P_{n+2}(x_{i},...,x_{n+3}+(-c)^{\frac{1}{2}},(-c)^{\frac{1}{2}},x_{1}+(-c)^{\frac{1}{2}},...,x_{i-3})\\
& \ \ =P_{n-i+3}(x_i,...,x_{n+2})[(-c)^{\frac{1}{2}}x_{n+3}[x_1+(-c)^{\frac{1}{2}}P_{i-4}(x_2,...,x_{i-3})\\
& \qquad +cP_{i-5}(x_3,...,x_{i-3})-(-c)^{\frac{3}{2}}P_{i-4}(x_{2},...,x_{i-3})+cx_{n+3}P_{i-4}(x_{2},...,x_{i-3})]\\
& \qquad + cP_{n-i+2}(x_i,...,x_{n+1})[(-c)^{\frac{1}{2}}x_{1}P_{i-4}(x_{2},...,x_{i-3})\\
& \qquad -(-c)^{\frac{3}{2}}P_{i-5} (x_{3},...,x_{i-3})]~,
\end{align*}
which can be rewritten as
\begin{align*}
&P_{n+2}(x_{i},...,x_{n+3}+(-c)^{\frac{1}{2}},(-c)^{\frac{1}{2}},x_{1}+(-c)^{\frac{1}{2}},...,x_{i-3})\\
& \qquad =P_{n-i+3}(x_i,...,x_{n+2})[[(-c)^{\frac{1}{2}}x_{1}x_{n+3}-(-c)^{\frac{3}{2}}]P_{i-4}(x_2,...,x_{i-3})\\
& \qquad \qquad +(-c)^{\frac{1}{2}}x_{n+3}cP_{i-5}(x_{3},...,x_{i-3})]\\
& \qquad \qquad + cP_{n-i+2}(x_{n+3},...,x_{i-3})[(-c)^{\frac{1}{2}}x_{1}P_{i-4}(x_{2},...,x_{i-3})\\
& \qquad \qquad -(-c)^{\frac{3}{2}}P_{i-5} (x_{3},...,x_{i-3})] \\
& \qquad = P_{n-i+3}(x_i,...,x_{n+2})[(-c)^{\frac{1}{2}}x_{n+3}]P_{i-4}(x_2,...,x_{i-3})x_1\\
& \qquad \qquad +cP_{i-5}(x_{3},...,x_{i-3})] - (-c)^{\frac{3}{2}}P_{i-4}(x_2,...x_{i-3})]\\
& \qquad \qquad + cP_{n-i+2}(x_{n+3},...,x_{i-3})[(-c)^{\frac{1}{2}}P_{i-4}(x_{2},...,x_{i-3})x_1\\
& \qquad \qquad +cP_{i-5} (x_{3},...,x_{i-3})]~.
\end{align*}
Using proposition~\ref{Frisess}b), that is
\begin{align*}
&P_{n+2}(x_{i},...,x_{n+3}+(-c)^{\frac{1}{2}},(-c)^{\frac{1}{2}},x_{1}+(-c)^{\frac{1}{2}},...,x_{i-3})\\
& \qquad =P_{n-i+3}(x_i,...,x_{n+2}) \\
& \qquad \qquad [(-c)^{\frac{1}{2}}x_{n+3}P_{i-3}(x_1,...,x_{i-3})-(-c)^{\frac{3}{2}}]P_{i-4}(x_2,...,x_{i-3})]\\
& \qquad \qquad +cP_{n-i+2}(x_{n+3},...,x_{i-3})[(-c)^{\frac{1}{2}}P_{i-3}(x_1,...,x_{i-3})]
\end{align*}
and using it again yields
\begin{align*}
&P_{n+2}(x_{i},...,x_{n+3}+(-c)^{\frac{1}{2}},(-c)^{\frac{1}{2}},x_{1}+(-c)^{\frac{1}{2}},...,x_{i-3})\\
& \qquad = P_{n-i+3}(x_i,...,x_{n+2})[(-c)^{\frac{1}{2}}P_{i-2}(x_{n+3},...,x_{i-3})]\\
& \qquad \qquad +cP_{n-i+2}(x_{n+3},...,x_{i-3})[(-c)^{\frac{1}{2}}P_{i-3} (x_{1},...,x_{i-3})]~.
\end{align*}
The equation can be rewritten as
\begin{align*}
&P_{n+2}(x_{i},...,x_{n+3}+(-c)^{\frac{1}{2}},(-c)^{\frac{1}{2}},x_{1}+(-c)^{\frac{1}{2}},...,x_{i-3})\\
& \qquad = (-c)^{\frac{1}{2}}[P_{n-i+3}(x_i,...,x_{n+2})P_{i-2}(x_{n+3},...,x_{i-3})\\
& \qquad \qquad +P_{n-i+2}(x_i,...,x_{n+1})P_{i-3} (x_{1},...,x_{i-3})]~,
\end{align*}
which, using once again proposition~\ref{Frisess}b), is
\begin{align*}
&P_{n+2}(x_{i},...,x_{n+3}+(-c)^{\frac{1}{2}},(-c)^{\frac{1}{2}},x_{1}+(-c)^{\frac{1}{2}},...,x_{i-3})\\
& \qquad =(-c)^{\frac{1}{2}}P_{n+1}(x_i,...,x_{i-3})~.
\end{align*}
Finally, replacing $P_{n+1}(x_i,...,x_{i-3})$ by $s$, we find
\[P_{n+2}(x_{i},...,x_{n+3}+(-c)^{\frac{1}{2}},(-c)^{\frac{1}{2}},x_{1}+(-c)^{\frac{1}{2}},...,x_{i-3})=(-c)^{\frac{1}{2}}s~.\]
Using the same arguments, one can show that 
\begin{align*}
&P_{n+3}(x_{i},...,x_{n+3}+(-c)^{\frac{1}{2}},(-c)^{\frac{1}{2}},x_{1}+(-c)^{\frac{1}{2}},...,x_{i-1})\\
& \qquad =(-c)^{\frac{1}{2}}P_{n+2}(x_{i},...,x_{i-1})\\
& \qquad =0~.\\
\end{align*}

This shows that the image of $\Gamma$ is a $c-$induced frieze of order $n+1$. Surjectivity follows from the fact that it is a sequence of equalities.
\end{proof}

We admit that the proof of this lemma is a lengthy calculation and so, a more conceptual argument, if possible, would be desirable.

\subsection{}
\theoremstyle{plain}\newtheorem*{Frnn+1}{Theorem}
\begin{Frnn+1}
Let $i \in \{1,2,3,\dots,n+3,n+4\}$ and fix $i_0 \in \{1,2,3,\dots,n+2\}$. For any repetitive $c-$frieze $f$ of order $n \geq 1$, there exists a $c-$induced frieze $f^{\prime}$ of order $n+1$, such that for all $1 \leq i \leq n+4$, we have \\
$
f'(i,i) = 
\begin{cases} 
f(i,i) & \ $if $i<i_0$; $\\
f(i,i) + (-c)^{\frac{1}{2}} & \ $if $i=i_0$; $\\
(-c)^{\frac{1}{2}} & \ $if $i=i_0+1$; $\\
f(i-1,i-1) + (-c)^{\frac{1}{2}} & \ $if $i=i_0+2$; $\\
f(i-1,i-1) & \ $if $i>i_0+2$. $
\end{cases}$

Furthermore, for any $c-$induced frieze $f^{\prime}$ of order $n \geq 2$, if we let $j_0 \in \{2,3,4,\dots,n+2\}$ be such that $f(j_0,j_0)=(-c)^{\frac{1}{2}}$, then there exists a repetitive $c-$frieze $f$ of order $n-1$ such that for all $1 \leq i \leq n+2$, we have \\
$
f(i,i) = 
\begin{cases} 
f'(i,i) & \ $if $i<j_0-1$; $\\
f'(i,i) - (-c)^{\frac{1}{2}} & \ $if $i=j_0-1$; $\\
f'(i+1,i+1) - (-c)^{\frac{1}{2}} & \ $if $i=j_0+1$; $\\
f'(i+1,i+1) & \ $if $i>j_0+1$. $
\end{cases}$
\end{Frnn+1}

\begin{proof}
This follows directly from the proof of lemma~\ref{nn+1}. The friezes are uniquely determined, because of corollary~\ref{consec}.
\end{proof}

\subsection{}
\theoremstyle{plain}\newtheorem*{monoint}{Corollary}
\begin{monoint}
Let $i \in \{1,2,3,\dots,n+3,n+4\}$ and fix $i_0 \in \{1,2,3,\dots,n+2\}$. If $(-c)$ is the square of an integer, then for any repetitive $c-$frieze $f$ of order $n \geq 1$, there exists a $c-$induced frieze of integers $f^{\prime}$ of order $n+1$, such that for all $1 \leq i \leq n+4$, we have \\
$
f'(i,i) = 
\begin{cases} 
f(i,i) & \ $if $i<i_0$; $\\
f(i,i) + (-c)^{\frac{1}{2}} & \ $if $i=i_0$; $\\
(-c)^{\frac{1}{2}} & \ $if $i=i_0+1$; $\\
f(i-1,i-1) + (-c)^{\frac{1}{2}} & \ $if $i=i_0+2$; $\\
f(i-1,i-1) & \ $if $i>i_0+2$. $
\end{cases}$

Furthermore, for any $c-$induced frieze of integers $f^{\prime}$ of order $n \geq 2$, if we let $j_0 \in \{2,3,4,\dots,n+2\}$ be such that $f(j_0,j_0)=(-c)^{\frac{1}{2}}$, then there exists a repetitive $c-$frieze of integers $f$ of order $n-1$ such that for all $1 \leq i \leq n+2$, we have \\
$
f(i,i) = 
\begin{cases} 
f'(i,i) & \ $if $i<j_0-1$; $\\
f'(i,i) - (-c)^{\frac{1}{2}} & \ $if $i=j_0-1$; $\\
f'(i+1,i+1) - (-c)^{\frac{1}{2}} & \ $if $i=j_0+1$; $\\
f'(i+1,i+1) & \ $if $i>j_0+1$. $
\end{cases}$
\end{monoint}

\begin{proof}
For the first part, by hypothesis, $x_1, x_2, x_3, \dots, x_{n+3}$ are all integers. Using the same transformations as in theorem~\ref{nn+1}, the resulting $c-$polygonal sequence of order $n+1$ consists only of integers, because $(-c)$ is the square of an integer. Then, by lemma~\ref{frisep} c), the associated friezes consist only of integers.

The arguments are similar for the second part.
\end{proof}

\theoremstyle{remark}\newtheorem*{Injectivity}{Remarks}
\begin{Injectivity}
As mentioned before, all friezes of positive integers in [CC73] have $1$ in their first row. Thus, it is possible to construct all $(-1)-$friezes of positive integers of order $n$ from $(-1)-$friezes of positive integers of order $n+1$, as mentioned in [CC73] p.179. 
\end{Injectivity}

\theoremstyle{remark}\newtheorem{ExampleNoSurjection}[exFrise]{Example}
\begin{ExampleNoSurjection}
Consider the following $(-4)-$frieze of positive integers of order $n=2$, with $s=t=8$:

\[\ldots \begin{array}{ccccccccccccccccc}
& 0 && 0 && 0 && 0 && 0 && 0 && 0 && 0\\
&& 1 && 1 && 1 && 1 && 1 && 1 && 1\\
& 1 && 6 && 6 && 1 && 16 && 1 && 6 && 6 \\
&& 2 && 32 && 2 && 12 && 12 && 2 && 32 && \\
& 8 && 8 && 8 && 8 && 8 && 8 && 8 && 8 \\
&& 0 && 0 && 0 && 0 && 0 && 0 && 0
\end{array} \ldots\]
\\
Using the transformation $\Gamma$ of theorem~\ref{nn+1}, it is possible to construct $(-4)-$ induced friezes of positive integers of order $n=3$ from the one given above. A resulting frieze, using $\psi_{-4}^{2}=(1,6,6,1,16)$ is:

\[\ldots \begin{array}{ccccccccccccccccc}
& 0 && 0 && 0 && 0 && 0 && 0 && 0 && 0\\
&& 1 && 1 && 1 && 1 && 1 && 1 && 1\\
& 3 && 6 && 6 && 1 && 18 && 2 && 3 && 6 \\
&& 14 && 32 && 2 && 14 && 32 && 2 && 14 &&\\
& 4 && 72 && 8 && 12 && 24 && 24 && 4 && 72 \\
&& 16 && 16 && 16 && 16 && 16 && 16 && 16 \\
& 0 && 0 && 0 && 0 && 0 && 0 && 0 && 0
\end{array} \ldots \]
\\
However, since the original $(-4)-$frieze of order $n=2$ is not $(-4)-$induced, it is not possible to find a $(-4)-$frieze of order $n=1$ with the arguments of theorem~\ref{nn+1}.
\end{ExampleNoSurjection}

\section*{Acknowledgements}
The authors thank Prof. Ibrahim Assem for his support, guidance and advice. The first and third authors gratefully acknowledge partial support from an undergraduate research fellowship from NSERC and from Prof. Ibrahim Assem. The second author gratefully acknowledges partial support from an undergraduate research fellowship from the Facult\'{e} des sciences of the Universit\'{e} de Sherbrooke and from Prof. Ibrahim Assem. 

\nocite{*}
\bibliography{c-friezesbib}
\bibliographystyle{alpha}


\end{document}